\documentclass[11pt,leqno]{amsart}
\usepackage{amsmath,amssymb,amsthm}
\usepackage{hyperref}
\newtheorem{theorem}{Theorem}[section]
\newtheorem{lemma}[theorem]{Lemma}

\newtheorem{proposition}[theorem]{Proposition}
\newtheorem{corollary}[theorem]{Corollary}
\theoremstyle{definition}
\newtheorem{definition}[theorem]{Definition}
\newtheorem{question}{Question}

\theoremstyle{remark}
\newtheorem{remark}[theorem]{Remark}
\numberwithin{equation}{section}
\def\fnote#1{\footnote}

\def\co{\operatorname{co}}

\def\ignora#1{}
\def\n3#1{\left\vert  \! \left\vert \! \left\vert \, #1 \, \right\vert \!
  \right\vert \! \right\vert }

\newcommand{\iten}{\ensuremath{\widehat{\otimes}_\varepsilon}}
\newcommand{\pten}{\ensuremath{\widehat{\otimes}_\pi}}

\title{Daugavet property in tensor product spaces}

\author [A. Rueda Zoca]{Abraham Rueda Zoca}
\address{Universidad de Granada, Facultad de Ciencias.
Departamento de An\'{a}lisis Matem\'{a}tico, 18071-Granada
(Spain)} \email{ abrahamrueda@.ugr.es}
\urladdr{\url{https://arzenglish.wordpress.com}}

\author[P. Tradacete]{Pedro Tradacete}
\address{Instituto de Ciencias Matem\'aticas (CSIC-UAM-UC3M-UCM)\\
Consejo Superior de Investigaciones Cient\'ificas\\
C/ Nicol\'as Cabrera, 13--15, Campus de Cantoblanco UAM\\
28049 Madrid, Spain.}
\email{pedro.tradacete@icmat.es}
\urladdr{\url{https://www.icmat.es/miembros/ptradacete/}}

\author[I. Villanueva]{Ignacio Villanueva}
\address{Universidad Complutense de Madrid, Departamento de An\'alisis Matemático y Matem\'atica Aplicada. Instituto de Matem\'atica Interdisciplinar-IMI. Instituto de Ciencias Matem\'aticas ICMAT, Madrid, Spain.}
\email{ignaciov@ucm.es}

\thanks{The research of first author was supported by MECD (Spain) FPU2016/00015, by MINECO (Spain) Grant MTM2015-65020-P and by Junta de Andaluc\'ia Grant FQM-0185. Second author gratefully acknowledges support of MINECO (Spain) through grants MTM2016-76808-P, MTM2016-75196-P, and the ``Severo Ochoa Programme for Centres of Excellence in R\&D'' (SEV-2015-0554). Third author was supported by MINECO (Spain) Grant MTM2017-88385-P,  QUITEMAD+-CM  (S2013/ICE- 2801) and the ``Severo Ochoa Programme for Centres of Excellence in R\&D'' (SEV-2015-0554).}



\begin{document}

\maketitle

\begin{abstract}
We study the Daugavet property in tensor products of Banach spaces. We show that $L_1(\mu)\iten L_1(\nu)$ has the Daugavet property when $\mu$ and $\nu$ are purely non-atomic measures. Also, we show that  $X\pten Y$ has the Daugavet property provided $X$ and $Y$ are $L_1$-preduals with the Daugavet property, in particular spaces of continuous functions with this property. With the same tecniques, we also obtain consequences about roughness in projective tensor products as well as the Daugavet property of projective symmetric tensor products.
\end{abstract}

\section{Introduction}

A Banach space $X$ is said to have the Daugavet property (DP) if every rank-one operator $T:X\longrightarrow X$ satisfies the equation
\begin{equation}\label{ecuadauga}
\Vert I+T\Vert=1+\Vert T\Vert,
\end{equation}
where $I$ denotes the identity operator. The previous equality is known as \textit{Daugavet equation} after I. Daugavet who proved in \cite{dau} that every compact operator on $\mathcal C([0,1])$ satisfies (\ref{ecuadauga}). Since then, several examples of Banach spaces enjoying the Daugavet property have appeared such as $\mathcal C(K)$ for a compact Hausdorff and perfect topological space $K$, $L_1(\mu)$ and $L_\infty(\mu)$ for a non-atomic measure $\mu$, and the space of Lipschitz functions $\operatorname{Lip(M)}$ over a metrically convex space $M$ (see \cite{ikw,kssw,werner} and the references therein for details). Moreover, in \cite[Lemma 2.1]{kssw} a characterisation of the Daugavet property in terms of the geometry of the slices of $B_X$ appeared (see below). This celebrated characterisation opened the door to understanding the many geometrical interpretations of the DP and has motivated a lot of research on the Daugavet equation ever since (see for instance \cite{bm2, bspw, ksw, kw}).

Once the DP has been understood for the classical Banach spaces, it is natural to study its stability under different combinations of these spaces. In this direction, it is important to understand its stability under tensor products. Given the preeminent position of the injective ($\epsilon$) and projective ($\pi$) norms as the smallest and largest respectively tensor norms, it becomes apparent the need to understand the stability of the DP under these norms. Indeed, in his 2001 survey paper \cite{werner}, D. Werner posed a list of open problems related to the Daugavet property. In particular,  \cite[Section 6, Question (3)]{werner} says:
\begin{center}
    \emph{If $X$ and/or $Y$ have the Daugavet property, what about their tensor products $X\iten Y$ and $X\pten Y$?}
\end{center}

Soon afterwards, V. Kadets, N. Kalton and D. Werner provided an example of a two dimensional complex Banach space $Y$ without the DP such that both $L_1^\mathbb C([0,1])\iten Y$ and  $L_\infty^\mathbb C([0,1])\pten Y^*$ fail the Daugavet property \cite[Theorem 4.2 and Corollary 4.3]{kkw}. Real counterexamples were given in \cite[Remark 3.13]{llr2} (the examples given there actually fail a weaker requirement than the Daugavet property). Therefore, the ``or'' part of the previous question was answered in the negative. 

In view of the preceding paragraph, the following question remains open:

\begin{question}\label{questiondproyect}
Let $X$ and $Y$ be Banach spaces with the Daugavet property. Do $X\iten Y$ and $X\pten Y$ have the Daugavet property?
\end{question}

It is well known that $L_1(\mu)\pten X=L_1(\mu,X)$ has the Daugavet property whenever $L_1(\mu)$ does. Concerning non-trivial positive results, we only know of two results. On the one hand, in \cite{br} it is proved, making a strong use of the theory of centralizer and function module representation of Banach spaces, that the projective tensor product of a Banach space without minimal $L$-summands with another non-zero Banach space has the Daugavet property. On the other hand, the first author proved in \cite{rueda} that $X\pten Y$ has the Daugavet property provided $X$ is a separable $L$-embedded Banach space with the Daugavet property and $Y$ is a non-zero Banach space with the metric approximation property. Anyway, a common denominator in both results is that, in order to get that $X\pten Y$ has the Daugavet property, only one of the spaces is required to enjoy the property. As a consequence of this fact, to the best of our knowledge, no positive result is known in the direction of Question \ref{questiondproyect}. 

The main results of this paper are two positive partial answers to  Question \ref{questiondproyect}:

\begin{theorem}\label{theo:L1itenL1}
Let $(\Omega_1,\Sigma_1,\mu_1)$ and $(\Omega_2,\Sigma_2,\mu_2)$ be measure spaces with purely non-atomic measures. Then the space $L_1(\Omega_1,\Sigma_1,\mu_1)\iten L_1(\Omega_2,\Sigma_2,\mu_2)$ has the Daugavet property.
\end{theorem}

\begin{theorem}\label{theo:maintheorem}
Let $X$ and $Y$ be two $L_1$-preduals. If $X$ and $Y$ have the Daugavet property, then so does $X\pten Y$. In particular, $C(K_1)\pten C(K_2)$ has the Daugavet property if $K_1$ and $K_2$ are compact spaces without isolated points.
\end{theorem}

The proof of Theorem \ref{theo:L1itenL1} is based on a discretization approach. It requires only the definition of the injective tensor norm and some fine measure theoretical arguments. Section 
\ref{section:inyedp} is devoted to this proof and it is totally self contained. 

\smallskip

By the duality of the injective and projective tensor norms, Theorems \ref{theo:L1itenL1} and \ref{theo:maintheorem} are somehow dual to each other. Therefore, it is not too surprising that it is possible to use similar measure theoretical arguments as in the proof of Theorem \ref{theo:L1itenL1} to prove that $C(K_1)\pten C(K_2)$ has the Daugavet property whenever $C(K_1)$ and $C(K_2)$ have it. However, a more abstract version of this proof allows us to extend the result to general $L_1$ preduals with the DP, and this is the proof we present in Section \ref{section:ODP}, which is again quite self contained. 

\smallskip

An inspection to this last proof together with the characterisation of the Daugavet property in $X\pten Y$ given in Proposition \ref{propmotiexten} point out the convenience of identifying a certain property about extension of bounded operators from $X$ to $Y^*$. Motivated by this, in Section \ref{section:ODP1} (see Definition \ref{defi:odp}), we introduce  the \textit{operator Daugavet property (ODP)}. Since ODP is a sufficient condition for Banach spaces $X$ and $Y$ in order to make $X\pten Y$ enjoy the DP (see Theorem \ref{theo:odpwerneranswer}), the rest of Section \ref{section:ODP1} is devoted to providing new examples of Banach spaces with the ODP. These include for instance, $L_1$-preduals with the Daugavet property, $L_1(\mu)$ spaces with non-atomic measures, or $\ell_\infty$-sums of spaces with the ODP.

\smallskip

In Section \ref{section:appodp}, we show how ODP can be applied to solve different isometric problems in the setting of tensor products. First, in Theorem \ref{theo:symmocta}, it is proved that if a Banach space $X$ has ODP, then all the projective symmetric tensor products $\widehat{\otimes}_{\pi,s,N} X$ have an octahedral norm (see definition below). In general, we are not able to get Daugavet property in such symmetric tensor product spaces because we lack a good description of norming sets for spaces of polynomials. However, making  use of the Dunford-Pettis property, we will prove in Proposition \ref{propo:tensosime} that $\widehat{\otimes}_{\pi,s,N} \mathcal C(K)$ has the Daugavet property whenever $K$ is a compact Hausdorff topological space without any isolated point and $N\in\mathbb N$. Let us point out that, to the best of our knowledge, the first (non-trivial) examples of projective symmetric tensor product spaces with the Daugavet property or with an octahedral norm are those given in Section \ref{section:appodp}. 

In this same section, we also use the ODP in order to get some consequences about roughness in projective tensor products. Indeed, we prove in Proposition \ref{propo2-ruda} that the norm of $X\pten Y$ is $2$-rough whenever $X$ has the ODP and $Y$ is non-zero. As an application, we derive consequences about stability of diameter two properties by injective tensor products of the form $L_1\iten X$. These are motivated by the question, posed in \cite[Question (b)]{aln}, about how diameter two properties are preserved by tensor product spaces.

\subsection{Terminology} 
We will consider only real Banach spaces. Given a Banach space $X$, we will denote the closed unit ball and the unit sphere of $X$ by $B_X$ and $S_X$ respectively. We will also denote by $X^*$ the topological dual of $X$. Given a bounded subset $C$ of $X$, $x^*\in X^*$ and $\alpha>0$, a \textit{slice of $C$} is given by
$$S(C,x^*,\alpha):=\{x\in C:x^*(x)>\sup x^*(C)-\alpha\}.$$

 A Banach space $X$ is said to have the \textit{Daugavet property} if every rank-one operator $T:X\longrightarrow X$ satisfies the equation
$$\Vert I+T\Vert=1+\Vert T\Vert,$$
where $I:X\longrightarrow X$ denotes the identity operator. It is known \cite{kssw} that a Banach space $X$ has the Daugavet property if, and only if, for every $\varepsilon>0$, every point $x\in S_X$ and every slice $S$ of $B_X$ there exists a point $y\in S$ such that $\Vert x+y\Vert>2-\varepsilon$. This characterisation will be freely used throughout the text without any explicit mention.

By an $L_1$-predual we will mean a Banach space $X$ such that $X^*=L_1(\mu)$ for certain measure $\mu$. We refer the reader to the seminal paper \cite{linds} for background on these spaces and the connection with norm preserving extension of operators. Also, we refer to \cite{bm} for background about $L_1$-preduals with the Daugavet property.

Given two Banach spaces $X$ and $Y$, we denote by $L(X,Y)$ the space of bounded linear operators $T:X\longrightarrow Y$. Also, we denote by $B(X,Y)$ the space of bounded bilinear maps $G:X\times Y\rightarrow \mathbb R$. Recall that the
\textit{projective tensor product} of $X$ and $Y$, denoted by
$X\pten Y$, is the completion of the algebraic tensor product $X\otimes Y$ under the norm given by
\begin{equation*}
   \Vert u \Vert :=
   \inf\left\{
      \sum_{i=1}^n  \Vert x_i\Vert\Vert y_i\Vert
      : u=\sum_{i=1}^n x_i\otimes y_i
      \right\}.
\end{equation*}
It follows easily from the definition  that $B_{X\pten Y}=\overline{\co}(B_X\otimes B_Y)
=\overline{\co}(S_X\otimes S_Y)$.
Moreover, given Banach spaces $X$ and $Y$, it is well known that
$(X\pten Y)^*=L(X,Y^*)=B(X,Y)$ \cite{DeFl}.

The \textit{injective tensor product} of $X$ and $Y$, denoted by $X \iten Y$, is the completion of $X\otimes Y$ under the norm given by
\begin{equation*}
   \Vert u\Vert:=\sup
   \left\{
      \sum_{i=1}^n \vert x^*(x_i)y^*(y_i)\vert
      : x^*\in S_{X^*}, y^*\in S_{Y^*}
   \right\},
\end{equation*}
where $u:=\sum_{i=1}^n x_i\otimes y_i$. Note that, in the above formula, $S_{X^*}$ and $S_{Y^*}$ can be replaced with norming sets for $X$ and $Y$ respectively. We refer the reader to \cite{DeFl,rya} for a detailed treatment of tensor product spaces.

\section{$L_1\iten L_1$ has the Daugavet property}\label{section:inyedp}

In this section, we prove Theorem \ref{theo:L1itenL1}. As mentioned in the introduction, the proof only requires the definition of the injective tensor product and measure theoretical reasonings. 

%\begin{theorem}\label{theo:L1itenL1}
%Let $(\Omega_1,\Sigma_1,\mu_1)$ and $(\Omega_2,\Sigma_2,\mu_2)$ be measure %spaces with purely non-atomic measures. The space %$L_1(\Omega_1,\Sigma_1,\mu_1)\hat\otimes_\varepsilon %L_1(\Omega_2,\Sigma_2,\mu_2)$ has the Daugavet property.
%\end{theorem}

\begin{proof}[Proof of Theorem \ref{theo:L1itenL1}]
For brevity, let $X=L_1(\Omega_1,\Sigma_1,\mu_1)\iten L_1(\Omega_2,\Sigma_2,\mu_2)$.

Let $\alpha \in X$ and $\varphi \in X^*$ with $\|\alpha\|_X=1=\|\varphi\|_{X^*}$, and let $\varphi\otimes\alpha$ denote the rank-one operator given by $\varphi\otimes\alpha(x)=\varphi(x)\alpha$ for $x\in X$. We will show that 
$$
\|I+\varphi\otimes\alpha\|=1+\|\varphi\otimes\alpha\|=2.
$$ 
To this end fix $\varepsilon>0$. Notice that, up a perturbation argument, there is no loss of generality in assuming that $\varphi$ is a norm-attaining functional. Let $\beta \in X$ with $\|\beta\|=1$ such that 
$$
\varphi(\beta)=1.
$$
Since simple functions are dense in any $L_1$ space, up to perturbation, we can assume without loss of generality that there are two collections of pairwise disjoint sets of finite measure $(A_i)_{i=1}^n\subset \Sigma_1$ and $(B_j)_{j=1}^n\subset \Sigma_2$, and scalars $(a_{ij})_{i,j=1}^n$, $(b_{ij})_{i,j=1}^n$ such that 
$$
\alpha =\sum_{i,j=1}^n a_{ij}\chi_{A_i}\otimes \chi_{B_j},
$$
and
$$
\beta =\sum_{i,j=1}^n b_{ij}\chi_{A_i}\otimes \chi_{B_j}.
$$
Also note that it follows immediately from the definition of the injective norm that the set 
$$
N=\{h_1\otimes h_2: \, h_i\in \operatorname{ext}(B_{L_\infty(\Omega_i,\Sigma_i,\mu_i)}), \,\text{ for }i=1,2\}
$$ is a norming subset of $X^*$; here $\operatorname{ext}(B_{L_\infty(\Omega_i,\Sigma_i,\mu_i)})$ denotes the set of extreme points of $B_{L_\infty(\Omega_i,\Sigma_i,\mu_i)}$, or in other words, $|h_i(x)|=1$ for $\mu_i$-almost every $x\in \Omega_i$. %(cf. \cite[PP. 46]{rya}).

We will need the following: 

\begin{lemma}\label{lemaLinfty}
For every $\delta>0$ there exist $(A'_i)_{i=1}^n\subset \Sigma_1$, $(B'_j)_{j=1}^n\subset\Sigma_2$ such that
\begin{enumerate}
\item $A'_i\subset A_i$, $\mu_1(A'_i)<\delta$,  for $1\leq i\leq n$.
\item $B'_j\subset B_j$, $\mu_2(B'_j)<\delta$, for $1\leq j\leq n$.
\item If we denote $$\beta'=\sum_{i,j=1}^n b_{ij} \frac{\mu_1(A_i)\mu_2(B_j)}{\mu_1(A'_i)\mu_2(B'_j)}\chi_{A'_i}\otimes\chi_{B'_j},$$ then we have $$\|\beta'\|_X=1\quad\textrm{ and }\quad\varphi(\beta')>1-\varepsilon.$$
\end{enumerate}
\end{lemma}

\begin{proof}
Fix $1\leq i\leq n$. Let $r_i=\sum_{j=1}^n b_{ij}\varphi(\chi_{A_i}\otimes \chi_{B_j})$. Note that $\sum_{i=1}^n r_i=\varphi(\beta)=1.$ For $f\in L_1(\Omega_1,\Sigma_1,\mu_1)$, let 
$$
\varphi_i(f)= \sum_{j=1}^n b_{ij}\mu_1(A_i)\varphi(f\chi_{A_i}\otimes \chi_{B_j}).
$$
Clearly, $\varphi_i$ is linear and 
$$
|\varphi_i(f)|\leq \|\varphi\|_{X^*} \sum_{j=1}^n |b_{ij}|\mu_1(A_i)\mu_2(B_j)\|f\|_{L_1}\leq C\|f\|_{L_1},
$$
for some finite $C$. Moreover, $\varphi_i(f)=0$ whenever $f\chi_{A_i}=0$. Hence, by Radon-Nikodym Theorem, there is $g_i\in L_\infty(A_i,\Sigma_1\cap A_i,\mu_1|_{A_i})$ such that
$$
\varphi_i(f)=\int_{A_i} g_i f d\mu_1.
$$
Since 
$$
\frac{1}{\mu_1(A_i)}\int_{A_i} g_i d\mu_1=\varphi_i\Big(\frac{\chi_{A_i}}{\mu_1(A_i)}\Big)=r_i,
$$
it follows that $g_i>r_i-\varepsilon/2n$ on a subset of $A_i$ with positive measure. Let $A'_i$ be such a set satisfying the additional requirement that $\mu_1(A'_i)<\delta$.
We have that
$$
\varphi_i\Big(\frac{\chi_{A'_i}}{\mu_1(A'_i)}\Big)=\frac{1}{\mu_1(A'_i)}\int_{A'_i} g_id\mu_1>r_i-\frac{\varepsilon}{2n}.
$$

Let now 
$$
\beta'=\sum_{i,j=1}^n b_{ij} \frac{\mu_1(A_i)}{\mu_1(A'_i)}\chi_{A'_i}\otimes\chi_{B_j}.
$$
It follows that
$$
\varphi(\beta')=\sum_{i,j=1}^n b_{ij} \frac{\mu_1(A_i)}{\mu_1(A'_i)}\varphi(\chi_{A'_i}\otimes\chi_{B_j})=\sum_{i=1}^n \varphi_i\Big(\frac{\chi_{A'_i}}{\mu_1(A'_i)}\Big)>\sum_{i=1}^n r_i-\frac{\varepsilon}{2n}= 1-\frac{\varepsilon}{2}.
$$
Moreover, we have 
\begin{align*}
\|\beta'\|_X&=\sup_{h_1\otimes h_2\in N} \langle h_1\otimes h_2,\beta'\rangle\\
&=\sup_{h_1\otimes h_2\in N}  \sum_{i,j=1}^n b_{ij} \frac{\mu_1(A_i)}{\mu_1(A'_i)}\int_{A'_i}h_1d\mu_1\int_{B_j}h_2d\mu_2\\
&=\sup_{\varepsilon_i,\sigma_j\in\{-1,+1\}}  \sum_{i,j=1}^n b_{ij}\varepsilon_i\sigma_j \mu_1(A_i)\mu_2(B_j)\\
&=\sup_{h_1\otimes h_2\in N}  \sum_{i,j=1}^n b_{ij} \int_{A_i}h_1d\mu_1\int_{B_j}h_2d\mu_2\\
&=\|\beta\|_X.
\end{align*}

Finally, if we make the same argument starting with $\beta'$ and interchanging the role of $i$ and $j$, then the result follows.
\end{proof}

Now, let $f\otimes g\in N$ be such that
\begin{equation}\label{ec:th41normavie}
\langle f\otimes g,\alpha\rangle>1-\varepsilon.
\end{equation}
Take 
\begin{equation}\label{ec:th41delta}
0<\delta<\frac{\varepsilon}{2\left(\underset{1\leq i\leq n}{\max}\mu_1(A_i)+\underset{1\leq j\leq n}{\max}\mu_2(B_j)\right)\sum_{i,j=1}^n|a_{ij}|},
\end{equation}
and let $(A'_i)_{i=1}^n\subset \Sigma_1$, $(B'_j)_{j=1}^n\subset\Sigma_2$ and $\beta'$ as given in Lemma \ref{lemaLinfty}. Let also $f'\otimes g'\in N$ be such that 
\begin{equation}\label{ec:teo41normanuevo}
\langle f'\otimes g',\beta'\rangle>1-\varepsilon.
\end{equation}

Now, let us define
$$
\tilde f(x)=
\left\{
\begin{array}{cc}
f'(x)  & \text{for }x\in \bigcup_{i=1}^n A'_i   \\
&\\
f(x)  & \text{elsewhere,}     
\end{array}
\right.
$$
$$
\tilde g(y)=
\left\{
\begin{array}{cc}
g'(y)  & \text{for }y\in \bigcup_{j=1}^n B'_j   \\
&\\
g(y)  & \text{elsewhere.}     
\end{array}
\right.
$$

First, note that by our choice of $\delta$ we have
\begin{align*}
\Big|\sum_{i,j=1}^n a_{ij}&\Big(\langle (f'-f)\otimes g,\chi_{A'_i}\otimes\chi_{B_j}\rangle+\langle f\otimes (g'-g),\chi_{A_i}\otimes \chi_{B'_j}\rangle\Big)\Big|\leq\\
&\leq \sum_{i,j=1}^n |a_{ij}|\Big(\int_{A'_i}|f'|+|f|d\mu_1\int_{B_j}|g|d\mu_2+\int_{A_i}|f|d\mu_1\int_{B'_j}|g'|+|g|d\mu_2\Big)\leq\\
&\leq \sum_{i,j=1}^n |a_{ij}|2(\mu_1(A'_i)\mu_2(B_j)+\mu_1(A_i)\mu_2(B'_j))\mathop{\leq}\limits^{\mbox{\scriptsize(\ref{ec:th41delta})}} \varepsilon.
\end{align*}

From the above estimate and taking into account that $\varphi(\beta')>1-\varepsilon$, it follows that
\begin{align*}
\|I+\varphi\otimes \alpha\| &\geq \|\beta'+\varphi(\beta')\alpha\|_X\\
&\geq \langle\tilde f\otimes \tilde g, \beta'+\varphi(\beta')\alpha\rangle\\
&= \sum_{i,j=1}^n \Big(b_{ij} \frac{\mu_1(A_i)\mu_2(B_j)}{\mu_1(A'_i)\mu_2(B'_j)}\langle\tilde f\otimes \tilde g,\chi_{A'_i}\otimes\chi_{B'_j}\rangle+\varphi(\beta') a_{ij}\langle\tilde f\otimes \tilde g,\chi_{A_i}\otimes \chi_{B_j}\rangle\Big)\\
&= \langle f'\otimes g',\beta'\rangle+ \varphi(\beta')(\sum_{i,j=1}^n a_{ij} \big( \langle f\otimes  g,\chi_{A_i}\otimes\chi_{B_j}\rangle\\
&+\langle (f'-f)\otimes g,\chi_{A'_i}\otimes\chi_{B_j}\rangle+\langle f\otimes (g'-g),\chi_{A_i}\otimes \chi_{B'_j}\rangle\big)\\
&\mathop{>}\limits^{\mbox{\scriptsize (\ref{ec:teo41normanuevo})}} 1-\varepsilon +\varphi(\beta')\big(\langle f\otimes g,\alpha\rangle-\varepsilon\big)\\
&\mathop{>}\limits^{\mbox{\scriptsize(\ref{ec:th41normavie})}}1-\varepsilon+(1-\varepsilon)(1-2\varepsilon).
\end{align*}

Since $\varepsilon>0$ was arbitrary, we get that $\|I+\varphi\otimes \alpha\| \geq2$ as claimed.
\end{proof}

%As we have pointed out at the Introduction, it remains open from \cite[Section 6, Question (3)]{werner} whether $X\iten Y$ has the Daugavet property provided $X$ and $Y$ have the Daugavet property. Theorem \ref{theo:L1itenL1} that the answer is affirmative for $L_1$ spaces.

\begin{remark}
Notice that the same idea works for the injective tensor product of any finite number of $L_1$ spaces with the Daugavet property.
\end{remark}

\section[$C(K)\pten C(K)$ has the DP]{The projective tensor product of $L_1$ preduals has the Daugavet Property}\label{section:ODP}

In this section we prove Theorem \ref{theo:maintheorem}. We start with the following Proposition, which is inspired by the characterisation of octahedrality of the norm appeared in \cite[Lemma 3.21]{lan}.

\begin{proposition}\label{propmotiexten}
Let $X$ and $Y$ be Banach spaces. The following assertions are equivalent:
\begin{enumerate}
    \item \label{propomotiexten1} $X\pten Y$ has the Daugavet property.
    \item \label{propomotiexten2} Given a finite-dimensional subspace $E$ of $X\pten Y$, a slice $S$ of $B_{X\pten Y}$, an operator $T\in L(X,Y^*)$ and $\varepsilon>0$ we can find $x\otimes y\in S$ and $G\in L(X,Y^*)$ such that $T=G$ on $E$, $G(x)(y)=G(x\otimes y)=\Vert T\Vert$ and $\Vert G\Vert\leq (1-\varepsilon)^{-1}\Vert T\Vert$.
\end{enumerate}
\end{proposition}

\begin{proof}
(1)$\Rightarrow$(2). Let $E$ be a finite-dimensional subspace of $X\pten Y$, $S$ a slice of $B_{X\pten Y}$, $T\in L(X,Y^*)$ and $\varepsilon>0$. Since $X\pten Y$ has the Daugavet property, then \cite[Lemma 2.8]{kssw} yields the existence of another slice $R\subseteq S$ of $B_{X\pten Y}$ such that
\begin{equation}\label{prop:l1orto}
\Vert e+\lambda z\Vert>(1-\varepsilon)(\Vert e\Vert+\vert \lambda\vert)\end{equation}
holds for every $e\in E$, every $z\in R$ and every $\lambda\in\mathbb R$. Since $R$ is a slice of $B_{X\pten Y}$ and $B_{X\pten Y}=\overline{\operatorname{co}}(S_X\otimes S_Y)$ then we can find $x\otimes y\in R\subseteq S$. Now, define the following functional
$$\begin{array}{ccc}
\varphi: E\oplus\mathbb R (x\otimes y)\subseteq X\pten Y & \longrightarrow & \mathbb R\\
e+\lambda x\otimes y & \longmapsto & T(e)+\Vert T\Vert \lambda.
\end{array}$$
We claim that $\Vert \varphi\Vert\leq \frac{\Vert T\Vert}{1-\varepsilon}$. In fact, given $e\in E$ and $\lambda\in\mathbb R$, we have
\[
\begin{split}
\varphi(e+\lambda x\otimes y)=T(e)+\lambda \Vert T\Vert\leq \Vert T\Vert (\Vert e\Vert+\vert\lambda\vert)\mathop{
\leq}\limits^{\mbox{(\ref{prop:l1orto})}} \frac{\Vert T\Vert}{1-\varepsilon}\Vert e+\lambda x\otimes y\Vert.
\end{split}
\]
So $\varphi\in (E\oplus \mathbb R(x\otimes y))^*$. By Hahn-Banach theorem we can extend $\varphi$ to an element $G\in (X\pten Y)^*=L(X,Y^*)$, which satisfies the desired requirements.

(2)$\Rightarrow$(1). Pick $z\in S_{X\pten Y}$, a slice $S$ of $B_{X\pten Y}$ and $\varepsilon>0$. Let us find $x\otimes y\in S$ such that
$$\Vert z+x\otimes y\Vert>2(1-\varepsilon).$$
To this aim, pick $E:=\operatorname{span}\{z\}$ and an operator $T\in S_{L(X,Y^*)}$ such that $T(z)=1$. By assumption, we can find an element $x\otimes y\in S$ and an operator $G\in L(X,Y^*)$ with $\Vert G\Vert\leq (1-\varepsilon)^{-1}$, $G(z)=T(z)=1$ and $G(x\otimes y)=1$. Now
$$
\Vert z+x\otimes y\Vert\geq \frac{G(z+x\otimes y)}{\|G\|}\geq 2(1-\varepsilon).
$$
The arbitrariness of $z$, $S$ and $\varepsilon$ implies that $X\pten Y$ has the Daugavet property, as desired.
\end{proof}

In view of the previous proposition it is clear that the Daugavet property on $X\pten Y$ is strongly related to the possibility of extending operators on $L(X,Y^*)$. Using this  and the celebrated work of J. Lindenstrauss \cite{linds}, we prove the second of our main results. %will get the main result of the paper, which gives a positive answer to Question \ref{questiondproyect} in the context of $L_1$ preduals.

%\begin{theorem}\label{teowernerL1pred}
%Let $X$ and $Y$ be two $L_1$-preduals with the Daugavet property. Then $X\pten Y$ enjoys the Daugavet property.
%\end{theorem}

\begin{proof}[Proof of Theorem \ref{theo:maintheorem}]
In order to prove that $X\pten Y$ has the Daugavet property pick an element $z\in S_{X\pten Y}$, a slice $S=S(B_{X\pten Y},B,\alpha)$ and $\varepsilon>0$, and let us find an element $x\otimes y\in S$ such that $\Vert z+x\otimes y\Vert>2-\varepsilon$.

Choose $\eta>0$ small enough so that $\left(2-3\eta-\eta\left( \frac{1+\eta}{1-\eta}\right)^2\right)\left(\frac{1-\eta}{1+\eta} \right)^2 >2-\varepsilon$. Pick a norm-one bilinear form $G$ such that $G(z)>1-\eta$ and $x_0\in S_X,y_0\in S_Y$ such that $G(x_0,y_0)>1-\eta$.

Choose $x'\in S_X$ and $y'\in S_Y$ such that $B(x',y')>1-\alpha$. From the definition of projective norm consider $n\in\mathbb N, x_1,\ldots, x_{ n}\in X$ and $y_1,\ldots, y_{ n}\in Y$ such that
\begin{equation*}%\label{theo:aproxitenso}
\left\Vert z-\sum_{i=1}^{ n} x_i\otimes y_i\right\Vert<\eta.
\end{equation*}
Define $E:=\operatorname{span}
\{x_1,\ldots, x_{ n}\}\subseteq X$. Since $X$ has the Daugavet property we can find from \cite[Lemma 2.8]{kssw} an element $x\in S(B_X,B(\cdot, y'),\alpha)$ (note that the previous set defines a slice of $B_X$ since $B$ is bilinear and continuous) such that
$$\Vert e+\lambda x\Vert>(1-\eta)(\Vert e\Vert+\vert\lambda\vert)$$
holds for every $e\in E$ and every $\lambda\in\mathbb R$.

Similarly define $F:=\operatorname{span}\{y_1,\ldots, y_{ n}\}\subseteq X$. Since $Y$ has the Daugavet property we can find from \cite[Lemma 2.8]{kssw} an element $y\in S(B_Y,B(x,\cdot),\alpha)$ (note that $B(x)\in Y^*$, so the previous set defines a slice of $B_Y$) such that
$$\Vert f+\lambda y\Vert>(1-\eta)(\Vert f\Vert+\vert\lambda\vert)$$
holds for every $f\in F$ and every $\lambda\in\mathbb R$.
Notice that $B(x)(y)>1-\alpha$ which means that $x\otimes y\in S$.

Define $\psi:E\oplus\mathbb R x\longrightarrow X$ by the equation
$$\psi(e+\lambda x):=e+\lambda x_0.$$
We claim that $\Vert \psi\Vert\leq \frac{1}{1-\eta}$. Indeed, given $e\in E$ and $\lambda\in\mathbb R$, we have
$$\Vert \psi(e+\lambda x)\Vert=\Vert e+\lambda x_0\Vert\leq \Vert e\Vert+\vert\lambda\vert\leq \frac{1}{1-\eta}\Vert e+\lambda x\Vert.$$
Now, since $X$ is an $L_1$ predual and $\psi$ is a compact operator (it is actually a finite-rank operator), we can find by \cite[Theorem 6.1, (3)]{linds} an extension $\varphi:X\longrightarrow X$ such that $\Vert \varphi\Vert\leq \frac{1+\eta}{1-\eta}$.

Similarly we can construct a bounded operator $\phi:Y\longrightarrow Y$ such that $\phi(f)=f$ for every $f\in F$, $\phi(y)=y_0$ and $\Vert \phi\Vert\leq \frac{1+\eta}{1-\eta}$. Now define the bilinear form $T(u,v):=G(\varphi(u),\phi(v))$ for every $u\in X, v\in Y$. We claim that $\Vert T\Vert\leq \left( \frac{1+\eta}{1-\eta}\right)^2$. Indeed, given $u\in B_X$ and $y\in B_Y$ we have
$$\vert G(\varphi(u),\phi(v))\vert\leq \Vert G\Vert\Vert \varphi(u)\Vert \Vert \phi(v)\Vert\leq \left( \frac{1+\eta}{1-\eta}\right)^2.$$
Furthermore,
\[\begin{split}
T(z) \geq \sum_{i=1}^{ n} G(\varphi(x_i),\phi(y_i))-\eta\Vert T\Vert & >\sum_{i=1}^{ n} G(x_i)(y_i)-\eta\left( \frac{1+\eta}{1-\eta}\right)^2\\
& >1-2\eta-\eta\left( \frac{1+\eta}{1-\eta}\right)^2.
\end{split}
\]
Also
$$T(x,y)=G(x_0,y_0)>1-\eta.$$
Finally
\[
\begin{split}
\Vert z+x\otimes y\Vert&  \geq \frac{T\left( x\otimes y+z\right)}{\Vert T\Vert}\\
& \geq \frac{1-\eta+T(z)}{\Vert T\Vert}\\
& \geq \frac{2-3\eta-\eta\left( \frac{1+\eta}{1-\eta}\right)^2}{\Vert T\Vert}\\
& \geq \left(2-3\eta-\eta\left( \frac{1+\eta}{1-\eta}\right)^2\right)\left(\frac{1-\eta}{1+\eta} \right)^2\\
& >2-\varepsilon.
\end{split}
\]
Since $\varepsilon>0$ was arbitrary then $X\pten Y$ has the Daugavet property, so we are done.
\end{proof}

\begin{remark}\label{remark:longtensorpredul1}
Notice that the same idea works for the projective tensor product of any finite number of $L_1$ preduals with the Daugavet property.
\end{remark}

\section{The operator Daugavet Property}\label{section:ODP1}

%In \cite[Section 6]{werner} it is asked whether the Daugavet property is preserved by taking projective tensor product from both factors. The previous theorem proves that the answer is affirmative when dealing with $L_1$-preduals.

In view of the proof of Theorem \ref{theo:maintheorem}, we will define and study in this section an operator version of the Daugavet property in the spirit of Proposition \ref{propmotiexten}. This property will allow us to obtain, in the next section, further results about the geometry of tensor products.

\begin{definition}\label{defi:odp}
Let $X$ be a Banach space. We will say that $X$ has the \textit{operator Daugavet property (ODP)} if, for every $x_1\ldots x_n\in S_X$, every slice $S$ of $B_X$ and every $\varepsilon>0$ there exists an element $x\in S$ such that, for every $x'\in B_X$, there exists an operator $T:X\longrightarrow X$ with $\Vert T\Vert\leq 1+\varepsilon$, $T(x)=x'$ and $\Vert T(x_i)-x_i\Vert<\varepsilon$ for every $i\in\{1,\ldots, n\}$.
\end{definition}

\begin{remark}\label{remaodpimplidp}
Notice that the ODP implies the Daugavet property. Indeed, given a Banach space $X$ with the ODP, fix $x\in S_X$, $\varepsilon>0$ and a slice $S$ of $B_X$. By hypothesis, there is an operator $T:X\longrightarrow X$ and a point $y\in S$ such that $\Vert T(x)-x\Vert<\varepsilon$, $\Vert T\Vert\leq 1+\varepsilon$ and $T(y)=x$. The existence of such operator implies that $\Vert x+y\Vert>\frac{2-\varepsilon}{1+\varepsilon}$. Nevertheless, we do not know of any example satisfying the Daugavet property but not the ODP.\end{remark}

From the proof of Theorem \ref{theo:maintheorem} the following result should be clear.

\begin{theorem}\label{theo:odpwerneranswer}
Let $X$ and $Y$ be two Banach spaces. If $X$ and $Y$ have the ODP then $X\pten Y$ has the Daugavet property.
\end{theorem}

We will devote the remainder of this section to give examples of Banach spaces with the ODP in order to enlarge the class of spaces where Theorem \ref{theo:odpwerneranswer} applies.

It is clear from the proof of Theorem \ref{theo:maintheorem} that $L_1$ predual spaces with the Daugavet property actually have the ODP. Let us see another classical space with the Daugavet property which actually enjoys the ODP.

\begin{proposition}\label{prop:L1odp}
Let $(\Omega,\Sigma,\mu)$ be a non-atomic measure space. Then $L_1(\mu)$ has the ODP.
\end{proposition}

\begin{proof}
Let us write $X=L_1(\mu)$ for short. Consider $x_1,\ldots, x_n\in B_X, \varepsilon>0$ and a slice $S=S(B_X,\varphi,\alpha)$ of $B_X$. We can assume with no loss of generality that $\varepsilon<\alpha$. To begin with, we can assume $\varphi\in X^*$ has norm one, and pick some $f_0\in S$. Now, since $f_0\in L_1(\mu)$ has $\sigma$-finite support, we can find $g\in L_\infty(\mu)$ such that 
$$
\varphi(f)= \int_\Omega fg d\mu,
$$
for every $f\in L_1(\mu)$ whose support is included in that of $f_0$.
In particular, this means that there exists $A\in \Sigma$ contained in the support of $f_0$, with $0<\mu(A)<\infty$, $\vert g(t)\vert>1-\varepsilon$ for every $t\in A$ and $\operatorname{sign}(g_{|A})$ constant.

Since $\mu$ does not contain any atom, we can assume with no loss of generality that there exists a subset $B\subseteq A$ which $\mu(B)>0$ and such that 
$$
\|x_i\chi_B\|<\frac{\varepsilon}{4},
$$
for every $i\in\{1,\ldots, n\}$.

Now, for every $i\in\{1,\ldots, n\}$, we can find a simple functions $x_i'\in S_X$ such that
\begin{equation}\label{L1aproxsimple}
\Vert x_i\chi_{\Omega\backslash B}-x_i'\Vert<\frac{\varepsilon}{4},
\end{equation}
where $x_i'=\sum_{j=1}^{m} a_{ij}\chi_{A_{j}}$ for suitable $m\in\mathbb N$, $a_{ij}\in\mathbb R$ and pairwise disjoint $A_{j}\in \Sigma$ with $B\cap A_j=\emptyset$ for $j\in\{1,\ldots, m\}$. 
%It follows that $\{A_{ij}: i\in\{1,\ldots, n\}, j\in\{1,\ldots, k_i\}\}$ is a collection of positive measure measurable sets. Also, since $\mu$ does not contain any atom, we can assume with no loss of generality that there exists a subset $B\subseteq A$ which $\mu(B)>0$ and such that $B\cap A_{ij}=\emptyset$ holds for every $i,j$. We can refine the collection $\{A_{ij}: i\in\{1,\ldots, n\}, j\in\{1,\ldots k_i\}\}$ to a collection of positive measure and pairwise disjoint sets $\{B_k: k\in\{1,\ldots, M\}\}$. 

Let also $f_B:=\frac{1}{\mu(B)}\operatorname{sign}(g_{|B})\chi_B$. Note that
$$\varphi(f_B)=\frac{1}{\mu(B)}\int_\Omega g(t)\operatorname{sign}(g_{|B})\chi_B\ d\mu=\frac{1}{\mu(B)}\int_B \vert g(t)\vert\ d\mu\geq 1-\varepsilon>1-\alpha,$$
from where $f_B\in S$.

Now, in order to prove that $X$ has the ODP, pick an element $x'\in B_X$ and let us construct an operator $T:X\longrightarrow X$ satisfying the desired requirements. To this end, consider the operator $T:X\longrightarrow X$ given by the equation
$$
T(f):=\sum_{j=1}^m \frac{1}{\mu(A_j)}\int_{A_j} f\ d\mu\ \chi_{A_j}+\int_B f\ d\mu\  \operatorname{sign}(g_{|B}) x'.
$$
It is clear from the disjointness of the sets $B, A_1,\ldots, A_m$ and the fact that $\Vert x'\Vert\leq 1$ that $\Vert T(f)\Vert\leq \Vert f\Vert$ holds for every $f\in X$. Furthermore, it is clear from the definition that $T(x_i')=x_i'$ , so
\[
\begin{split}
\Vert T(x_i)-x_i\Vert& \leq  \Vert T(x_i-x_i')\Vert+\Vert x_i'-x_i\Vert\\
& \leq \Vert T\Vert \Vert x_i-x_i'\Vert+\Vert x_i-x_i'\Vert\\
& \mathop{\leq}\limits^{\mbox{
(\ref{L1aproxsimple})}}\varepsilon.\end{split}
\]
On the other hand,
$$T(f_B)=\frac{1}{\mu(B)}\operatorname{sign}(g_{|B})^2\int_B \chi_B\ d\mu x'=x',$$
 and the proof is finished.
\end{proof}

\begin{remark}\label{remark:L1}
\begin{enumerate}
    \item Notice that an adaptation of the proof of Proposition \ref{prop:L1odp} yields that $L_1(\mu,X)$ has the ODP whenever $\mu$ does not contain any atom regardless of the Banach space $X$.
    \item The fact that $L_1$-preduals with the Daugavet property have the ODP can be considered in part as a consequence of the fact that these are injective Banach spaces. The situation with $L_1(\mu)$ spaces can also be regarded as somehow similar, since these spaces are in turn injective as Banach lattices (cf. \cite{lotz, mn}).
\end{enumerate}

\end{remark}

%\textbf{TBC: DOS REMARKS: 1º NOTAR QUE FUNCIONA PARA ESPACIOS VECTOR VALUADOS CUANDO EL ESPACIO DE LLEGADA TIENE LA RNP. 2º TRATAR DE EXPLICAR QUE ESTE ESPACIO GOZA DE UNA CIERTA PROPIEDAD DE INYECTIVIDAD COMO RETICULO DE BANACH, LO QUE LO HACE NATURAL COMO EL EJEMPLO DE PREDUALES DE L1. A\~NADIR TANTAS REFERENCIAS COMO SEAN NECESARIAS.}

It is known that if $X$ and $Y$ have the Daugavet property then so does $X\oplus_\infty Y$ \cite[Theorem 1]{woj}. The following proposition proves that the same holds for the ODP.

\begin{proposition}
Let $X$ and $Y$ be two Banach spaces with ODP. Then $X\oplus_\infty Y$ has ODP.
\end{proposition}

\begin{proof}
Let $(x_1,y_1),\ldots, (x_n,y_n)\in S_{X\oplus_\infty Y}$, $\varepsilon>0$, and a slice
$$
R=S(B_{X\oplus_\infty Y},(x^*,y^*),\alpha).
$$
Let us find $(x,y)\in R$ such that, for every $(x',y')\in B_{X\oplus_\infty Y}$, there exists an operator $T:X\oplus_\infty Y\longrightarrow X\oplus_\infty Y$ satisfying the desired requirements. 

Since $X$ has the ODP, there exist an element $x\in \{z\in B_X: x^*(z)>\Vert x^*\Vert-\frac{\alpha}{2}\}$ (which is a slice of $B_X$) such that, for every $x'\in B_X$, there exists an operator $G:X\longrightarrow X$ with $\Vert G\Vert\leq 1+\varepsilon$, $\Vert G(x_i)-x_i\Vert\leq \varepsilon$ and $G(x)=x'$.

Repeating the same argument on the factor $Y$ find $y\in \{z\in B_Y: y^*(z)>\Vert y^*\Vert-\frac{\alpha}{2}\}$ (which is a slice of $B_Y$) such that, for every $y'\in B_Y$, there exists an operator $S:Y\longrightarrow Y$ with $\Vert S\Vert\leq 1+\varepsilon$, $\Vert S(y_i)-y_i\Vert\leq \varepsilon$ and $S(y)=y'$. 

Consider $(x,y)\in B_{X\oplus_\infty Y}$. Then
$$(x^*,y^*)( x, y)= x^*(x)+ y^*(y)>\Vert x^*\Vert+\Vert y^*\Vert-\alpha=1-\alpha,$$
which means that $(x,y)\in R$. In order to finish the proof let us show that $(x,y)$ satisfies the desired requirements. To this end, pick $(x',y')\in B_{X\oplus_\infty Y}$. Note that $\Vert(x',y')\Vert=\max\{\Vert x'\Vert,\Vert y'\Vert\} = 1$, which means that both $\Vert x'\Vert$ and $\Vert y'\Vert$ are less than or equal to $1$. Consider $G:X\longrightarrow X$ and $S:Y\longrightarrow Y$ with the properties described above associated to $x'$ and $y'$ respectively, and define $T:X\oplus_\infty Y \longrightarrow X\oplus_\infty Y$ given by $T(u,v)=(G(u),S(v))$ for every $u\in X, v\in Y$. First of all, given $u\in X, v\in Y$, we get
$$\Vert T(u,v)\Vert=\max\{\Vert G(u)\Vert,\Vert S(v)\Vert\}\leq (1+\varepsilon)\max\{\Vert u\Vert,\Vert v\Vert\}=(1+\varepsilon)\Vert (u,v)\Vert_\infty,$$
so $\Vert T\Vert\leq 1+\varepsilon$. Moreover,
$$
\Vert T(x_i,y_i)-(x_i,y_i)\Vert=\max\{\Vert G(x_i)-x_i\Vert,\Vert S(y_i)-y_i\Vert\}\leq\varepsilon.
$$
Finally notice that
$$T(x,y)=(x',y'),$$
which finishes the proof.
\end{proof}

In order to summarise the content of the section and to relate it with Question \ref{questiondproyect} we get the following corollary.

\begin{corollary}\label{cor:resumen}
Let $X$ and $Y$ be two Banach spaces with the Daugavet property. Then $X\pten Y$ has the Daugavet property if $X$ and $Y$ satisfy any of the following requirements:
\begin{enumerate}
    \item To be an $L_1$-predual.
    \item To be an $\ell_\infty$ sum of $L_1(\mu)$ spaces.
    \item To be an $\ell_\infty$ sum of an $L_1(\mu)$ space and an $L_1$-predual.
\end{enumerate}
\end{corollary}

\section{Further consequences of the ODP}\label{section:appodp}

In this section we obtain further connections between the ODP and different questions about the geometry of tensor product spaces.

\subsection{Daugavet property in the projective symmetric tensor product}

Given a Banach space $X$, we define the
\textit{($N$-fold) projective symmetric tensor product} of $X$, denoted by
$\widehat{\otimes}_{\pi,s,N} X$, as the completion of the space
$\otimes^{s,N}X$ under the norm
\begin{equation*}
   \Vert u\Vert:=\inf
   \left\{
      \sum_{i=1}^n \vert \lambda_i\vert \Vert x_i\Vert^N :
      u:=\sum_{i=1}^n \lambda_i x_i^N, n\in\mathbb N, x_i\in X
   \right\}.
\end{equation*}
The dual, $(\widehat{\otimes}_{\pi,s,N} X)^*=\mathcal P(^N X)$, is
the Banach space of $N$-homogeneous continuous polynomials on $X$, and notice that $B_{\widehat{\otimes}_{\pi,s,N} X}=\overline{\co}(\{x^N:x\in S_X\})$ (see \cite{flo} for background).

As far as we are concerned, no non-trivial example of projective symmetric tensor product with the Daugavet property is known. In the sequel, we will provide one such example using  the ODP. In order to do so, let us introduce a bit of notation. Recall that the norm of a Banach space $X$ is said to be \textit{octahedral} if, whenever $Y$ is a finite-dimensional subspace of $X$ and $\varepsilon>0$, there exists $x\in S_X$ such that
$$\Vert y+\lambda x\Vert>(1-\varepsilon)(\Vert y\Vert+\vert\lambda\vert)$$
holds for every $y\in Y$ and every $\lambda\in\mathbb R$. Daugavet property implies octahedrality by \cite[Lemma 2.8]{kssw}, but the converse is not true as the norm of $\ell_1$ is octahedral. Now we can prove the following result.

\begin{theorem}\label{theo:symmocta}
Let $X$ be a Banach space with the ODP and let $N\in\mathbb N$. Then the $N$-fold symmetric projective tensor product, $\widehat{\otimes}_{\pi, s, N} X$, has an octahedral norm.
\end{theorem}

\begin{proof}In order to save notation define $Y:=\widehat{\otimes}_{\pi, s, N} X$. In view of \cite[Proposition 2.1]{hlp} it is enough to prove that, given $z_1,\ldots, z_n\in S_Y$ and $\varepsilon>0$ we can find $x\in S_X$ such that
$$\left\Vert z_i+x^N\right\Vert>2-\varepsilon$$
holds for every $i\in\{1,\ldots, n\}$. Fix $i\in\{1,\ldots, n\}$ and choose a norm-one polynomial $P_i$ such that $P_i(z_i)=1$. Furthermore, since $B_Y=\overline{\operatorname{co}}(\{x^N:x\in S_X\})$ choose $v_i\in S_X$ such that $P_i(v_i)>1-\varepsilon$. Furthermore, find $\sum_{j=1}^{n_i}\lambda_{ij}x_{ij}^N\in \operatorname{co}(\{x^N:x\in S_X\})$ satisfying that
\begin{equation}\label{ecuasimapprox}
   \left\Vert  z-\sum_{j=1}^{n_i}\lambda_{ij}
x_{ij}^N\right\Vert<\varepsilon.
\end{equation}
Since $X$ has the ODP we can find an element $x\in S_X$ and operators $\varphi_i:X\longrightarrow X$ such that 
\begin{equation}\label{ecuasimpuntoscerca}\Vert \varphi_i(x_{ij})-x_{ij}\Vert<\varepsilon
\end{equation}
holds for every $j\in\{1,\ldots, n_i\}$, $\varphi_i(x)=v_i$ and $\Vert \varphi_i\Vert\leq 1+\varepsilon$. Now define $Q_i:=P_i\circ \varphi_i:X\longrightarrow \mathbb R$. Notice that $Q_i$ is a $N$-homogeneous polynomial. In fact, if we denote by $\hat{P_i}$ the $N$-linear form associated to $P_i$ (i.e. $P_i(z)=\hat{P_i}(z,z,\ldots z)$, then $\hat Q_i(z_1,\ldots, z_N)=\hat P_i(\varphi_i(z_1),\varphi_i(z_2),\ldots, \varphi_i(z_N))$, which is an $N$-linear form because of the linearity and continuity of $\varphi_i$). Furthermore, in order to estimate the polynomial norm of $ Q_i$, pick $x\in X$. Then
$$\vert Q_i(x)\vert=\vert P_i(\varphi_i(x))\vert\leq \Vert P_i\Vert \Vert \varphi_i(x)\Vert^N\leq \Vert \varphi_i\Vert^N \Vert x\Vert^N\leq (1+\varepsilon)^N\Vert x\Vert^N.$$
So $\Vert Q_i\Vert\leq (1+\varepsilon)^N$. Furthermore
\[
\begin{split}Q_i(z_i)& \geq \sum_{j=1}^{n_i}\lambda_{ij}Q_i(x_{ij})-
(1+\varepsilon)^N\left\Vert z-\sum_{j=1}^{n_i}\lambda_{ij}
x_{ij}^N\right\Vert\\
& \mathop{>}\limits^{\mbox{\tiny(\ref{ecuasimapprox})}}\sum_{j=1}^{n_i} \lambda_{ij} P_i(x_{ij})-\Vert P_i\Vert\Vert \varphi_i(x_{ij})-x_{ij}\Vert^N-(1+\varepsilon)^N
\varepsilon^N\\
& \mathop{>}\limits^{\mbox{\tiny(\ref{ecuasimpuntoscerca})}}\sum_{j=1}^{n_i}\lambda_{ij}P_i(x_{ij})-\varepsilon^N-
(1+\varepsilon)^N\varepsilon\\
& \mathop{>}\limits^{\mbox{\tiny(\ref{ecuasimapprox})}} P_i(z_i)-
\varepsilon-\varepsilon^N-(1+\varepsilon)^N\varepsilon^N\\
& =1-
\varepsilon-\varepsilon^N-(1+\varepsilon)^N\varepsilon^N.
\end{split}
\]
Moreover
$$Q_i(x)=P_i(v_i)>1-\varepsilon.$$
Hence
$$\left\Vert z_i+x^N\right\Vert\geq \frac{Q_i(z_i+x^N)}{\Vert Q_i\Vert}>\frac{2-2\varepsilon-\varepsilon^N-(1
+\varepsilon)^N\varepsilon^N}{(1+\varepsilon)^N}.$$
Since $i$ and $\varepsilon$ were arbitrary we conclude that the norm of $Y$ is octahedral, as desired.
\end{proof}

\begin{remark}\label{remarkdaugasym}
 Let $X$ be a Banach space with the ODP. Given $x_1,\ldots, x_n\in S_X, x'\in B_X$ and $\varepsilon>0$, consider the set $A$ of those $x\in B_X$ for which there exists a bounded operator $\varphi:X\longrightarrow X$ such that $\Vert \varphi(x_i)-x_i\Vert<\varepsilon, \varphi(x)=x'$ and $\Vert \varphi\Vert\leq 1+\varepsilon$. Then, in order to ensure that $\widehat{\otimes}_{\pi,s,N}X$ has the Daugavet property by an adaptation of the proof of Theorem \ref{theo:symmocta} we need to guarantee that the set $\{x^N:x\in A\}$ is norming for $\mathcal P(^N X)$ for every $x_1,\ldots, x_n,x'\in B_X$ and every $\varepsilon>0$.
\end{remark}

Although we do not know whether the property exhibited in the preceding remark holds in general for every ODP space, we will prove in the following proposition that this is the case for spaces of continuous functions.

\begin{proposition}\label{propo:tensosime}
Let $K$ be a compact Hausdorff space without isolated points. Then $\widehat{\otimes}_{\pi,s,N}\mathcal C(K)$ has the Daugavet property.
\end{proposition}

In order to prove the Proposition we need the following lemma.

\begin{lemma}\label{lemac(K)}
Let $K$ be a compact Hausdorff space without isolated points. Let $E$ be a finite-dimensional subspace of $\mathcal C(K)$ and $\varepsilon>0$. Then the set 
$$A:=\{g\in S_{\mathcal C(K)}:\Vert f+\lambda g\Vert>(1-\varepsilon)(\Vert f\Vert+\vert\lambda\vert)  \forall f\in E,\lambda\in\mathbb R\}$$
is weakly sequentially dense in $B_{\mathcal C(K)}$.
\end{lemma}

\begin{proof} Write $X=\mathcal C(K)$ for shorten. Let $0<\delta<\frac{\varepsilon}{4}$. Pick $f_1,\ldots, f_k$ to be a $\delta$-net in $S_E$ and $h\in S_X$. Let us find a sequence $\{g_n\}\in S_X$ such that $\{g_n\}\rightarrow h$ in the weak topology of $B_X$ and that $\Vert f_i+g_n\Vert>2-2\delta$ holds for every $i\in\{1,\ldots, k\}$ and every $n\in\mathbb N$. This is enough in view of the proof of \cite[Lemma 2.8]{kssw}. In order to do so consider, for every $i\in\{1,\ldots, k\}$, the set
$$A_i:=\{t\in K: \vert f_i(t)\vert>1-\delta\}.$$
Note that every $A_i$ is an (infinite) open subset of $K$. Now, making an inductive argument we can find, for every $i\in\{1,\ldots, k\}$, a sequence of non-empty open sets $V_n^i$ such that $\overline{V}_n^i\subseteq A_i$ for every $i\in\{1,\ldots, k\}$ and such that
$$\overline{V}_n^i\cap \overline{V}_m^j=\emptyset$$
holds for every $n,m\in\mathbb N$ and $i,j\in\{1,\ldots, k\}$ such that either $n\neq m$ or $i\neq j$. Now select, for every $i\in\{1,\ldots, k\}$ and every $n\in\mathbb N$ a point $t_n^i\in V_n^i$. Making use of Urysohn lemma we can construct, for every $n\in\mathbb N$, a function $g_n\in S_X$ such that $g_n(t_n^i)=\operatorname{sign}(f_i(t_n^i))f_i(t_n^i)$ and $g_n=h$ on $K\setminus\bigcup\limits_{i=1}^k V_n^i$. It is clear that the (bounded) sequence $\{g_n\}$ converges pointwise to $h$, so $\{g_n\}$ converges weakly to $h$ . Furthermore,
$$\Vert f_i+g_n\Vert\geq \vert f_i(t_n^i)+g(t_n^i)\vert=2\vert f_i(t_n^i)\vert>2-2\delta.$$
So the lemma is proved.
\end{proof}

\begin{proof}[Proof of Proposition \ref{propo:tensosime}]
Let $X=\mathcal C(K)$. According to Remark \ref{remarkdaugasym} we will prove that, given a finite-dimensional subspace $E$ of $X$, $x'\in B_X$ and a positive $\varepsilon>0$, if we define
$$A:=\{g\in S_X:\Vert f+\lambda g\Vert>(1-\varepsilon)(\Vert f\Vert+\vert\lambda\vert)  \forall f\in E,\lambda\in\mathbb R\},$$
then $\{g^N:g\in A\}$ is norming for $\mathcal P(^N X)$. Note that if the assertion were proved then, given $x\in A$, we could construct, by a similar argument to that of the proof of Theorem \ref{theo:maintheorem}, an operator $\varphi:X\longrightarrow X$ such that $\varphi(e)=e$ for every $e\in E$, $\varphi(x)=x'$ and $\Vert \varphi\Vert\leq 1+\varepsilon$ since $X$ is an $L_1$- predual. So the Proposition would follow by an application of Remark \ref{remarkdaugasym}.

Hence, in order to prove that $\{g^N:g\in A\}$ is norming for $\mathcal P(^N X)$, pick a norm-one polynomial $P\in \mathcal P(^N X)$, a positive $\varepsilon>0$ and a point $y\in S_Y$ such that $P(y)>1-\varepsilon$. By Lemma \ref{lemac(K)} we get that $A$ is weakly sequentially dense in $B_X$, so we can find a sequence of points in $A$, say $\{g_n\}$, such that $\{g_n\}\rightarrow y$ in the weak topology of $B_X$. Now, since $X$ has the Dunford-Pettis property \cite[Theorem 5.4.5]{alka} then $P(g_n)\rightarrow P(y)>1-\varepsilon$ \cite[Corollary 5.1]{gjl}, so we can find $n\in\mathbb N$ such that $P(g_n)>1-\varepsilon$. Since $g_n\in A$, the arbitrariness of $P$ and $\varepsilon$ proves the fact that $\{g^N: g\in A\}$ is norming for $\mathcal P(^N X)$, and the proposition is proved.
\end{proof}

\begin{remark}
To the best of our knowledge, Proposition \ref{propo:tensosime} provides the first non-trivial example of a projective symmetric tensor product enjoying the Daugavet property.
\end{remark}

\begin{remark}
Note that Lemma \ref{lemac(K)} does not hold for general Banach spaces with the Daugavet property. Indeed, in \cite[Theorem 2.5]{kw} an example of a Banach space with the Daugavet property and the Schur property is exhibited.
\end{remark}

\subsection{2-roughness in projective tensor product}

Let $X$ be a Banach space. Recall that the norm of $X$ is said to be \textit{$\varepsilon$- rough} if 
$$\limsup\limits_{\Vert h\Vert\rightarrow 0} \frac{\Vert x+h\Vert+\Vert x-h\Vert-2\Vert x\Vert}{\Vert h\Vert}\geq \varepsilon$$
holds for every $x\in X$.

Note that rough norms are ``uniformly non-Fr\'echet differentiable". See \cite[Chapter 1]{dgz} for background on rough norms.

Our aim will be to prove the following result.

\begin{proposition}\label{propo2-ruda}
Let $X$ be a Banach space with the ODP. Then the norm of $X\pten Y$ is $2$-rough for every non-zero Banach space $Y$.
\end{proposition}

In order to prove the proposition we need the following reformulation of 2-roughness, which will be useful in the sequel. The proof (1)$\Longleftrightarrow$(2) is \cite[Proposition I.1.11]{dgz} whereas that of (1)$\Longleftrightarrow$(3) is \cite[Lemma 3.1]{hlp}.

\begin{lemma}\label{lema:cara2-ruda}
Let $X$ be a Banach space. The following assertions are equivalent:
\begin{enumerate}
    \item The norm of $X$ is $2$-rough.
    \item Every weak-star slice of $B_{X^*}$ has diameter two.
    \item $X$ is \textit{locally octahedral} (LOH), that is, for every $x\in S_X$ and every $\varepsilon>0$ there exists $y\in S_X$ such that $\Vert x\pm y\Vert>2-\varepsilon$.
\end{enumerate}
\end{lemma}

\begin{proof}[Proof of Proposition \ref{propo2-ruda}]
In order to prove that the norm of $X\pten Y$ is $2$-rough, we will make use of Lemma \ref{lema:cara2-ruda} (3). Pick an element $z\in S_{X\pten Y}$ and $\varepsilon>0$, and let us find an element $x\otimes y\in S_{X\pten Y}$ such that $\Vert z\pm x\otimes y\Vert>2-\varepsilon$.

Choose $\eta>0$ small enough so that $\frac{2-\eta(4+\eta)}{1+\eta}>2-\varepsilon$. Pick a norm-one bilinear form $G$ such that $G(z)>1-\eta$ and $x_0\in S_X,y_0\in S_Y$ such that $G(x_0,y_0)>1-\eta$.

From the definition of projective norm consider $n\in\mathbb N, x_1,\ldots, x_{ n}\in S_X$, $y_1,\ldots, y_{ n}\in S_Y$ and $\lambda_1,\ldots, \lambda_n\in [0,1]$ such that $\sum_{i=1}^n \lambda_i=1$ and 
\begin{equation*}%\label{theo:aproxitenso2}
\left\Vert z-\sum_{i=1}^{ n} \lambda_i x_i\otimes y_i\right\Vert<\eta.
\end{equation*}
Since $X$ has the ODP we can find $x\in S_X$ and  $\psi^\pm:X\longrightarrow X$ such that $\Vert \psi^\pm(x_i)-x_i\Vert<\eta, \psi^\pm(x)=\pm x_0$ and 
$ \Vert \psi^\pm\Vert\leq 1+\eta$.

Now, define the bilinear forms $T^\pm(u,v):=G(\psi^\pm(u),v)$ for $u\in X, v\in Y$. We claim that $\Vert T^\pm \Vert\leq 1+\eta$. Indeed, given $u\in B_X$ and $y\in B_Y$ we have
$$
\vert G(\psi^\pm(u),v)\vert\leq \Vert G\Vert\Vert \psi^\pm(u)\Vert \Vert v\Vert\leq 1+\eta.
$$
Furthermore,
\[\begin{split}
T^\pm(z) & \geq \sum_{i=1}^{ n} \lambda_i G(\psi^\pm(x_i),y_i)-\eta\Vert T^\pm\Vert \\
& \geq \sum_{i=1}^{ n} \lambda_i G(x_i)(y_i)-\sum_{i=1}^n \lambda_i \Vert G\Vert \Vert \psi^\pm (x_i)-x_i\Vert\Vert y_i\Vert-\eta(1+\eta)\\
& >1-\eta(3+\eta).
\end{split}
\]
Also
$$T^\pm(\pm x,y)=G(x_0,y_0)>1-\eta.$$
Finally
\[
\begin{split}
\Vert z\pm x\otimes y\Vert&  \geq \frac{T^\pm\left(\pm x\otimes y+z\right)}{\Vert T^\pm\Vert}\\
& \geq \frac{1-\eta+T^{\pm}(z)}{1+\eta}\\
& \geq \frac{2-\eta(4+\eta)}{1+\eta}\\
&  >2-\varepsilon.
\end{split}
\]
Since $\varepsilon>0$ was arbitrary then $X\pten Y$ is LOH, and we are done.
\end{proof}

\begin{remark}
An examination of the previous proof yields that, in Proposition \ref{propo2-ruda}, if we localise $z$ in a given slice then we can get the \textit{local diametral diameter two property} (see \cite{blr2} for definition and backgroud). We do thank Johann Langemets for pointing out to the authors this improvement.
\end{remark}

In order to obtain some consequences from Proposition \ref{propo2-ruda}, let us introduce a bit of notation. Given a Banach space $X$, recall that $X$ is said to have the \textit{slice diameter two property (slice-D2P)} if every slice of the unit ball of $B_X$ has diameter two. If $X$ is itself a dual Banach space, then $X$ is said to have the \textit{weak-star slice diameter two property ($w^*$-slice-D2P)} if every weak-star slice of $B_X$ has diameter two. See \cite{aln, blr, llr} and references therein for background on diameter two properties.

Taking into account the duality $L(X,Y^*)=(X\pten Y)^*$ and Lemma \ref{lema:cara2-ruda} we get the following result.

\begin{corollary}
Let $X$ be a Banach space with the ODP. Then $L(X,Y^*)$ has the $w^*$-slice-D2P.
\end{corollary}

Let us end with an application of Proposition \ref{propo2-ruda} to the study of the slice diameter two property in injective tensor products. 

\begin{corollary}\label{cor:sliced2pinyec}
Let $(\Omega,\Sigma,\mu)$ be a measure space and assume that $\mu$ does not contain any atom. Let $X$ be a Banach space such that $X^*$ has the RNP. Then $L_1(\mu)\iten X$ has the slice-D2P.
\end{corollary}

\begin{proof}
Notice that $(L_1(\mu)\iten X)^*=L_\infty(\mu)\pten X^*$ by \cite[Theorem 5.33]{rya}, so its norm is $2$-rough by Proposition \ref{propo2-ruda}. Consequently, $L_1(\mu)\iten X$ has the slice-D2P by \cite[Theorem 3.3]{hlp}.
\end{proof}

In \cite[Question b)]{aln} it is asked how are the diameter two properties, in general, preserved by tensor product spaces. Corollary \ref{cor:sliced2pinyec} yields new examples of injetive tensor products with the slice-D2P different from those obtained in \cite[Theorem 5.3]{abr} and of \cite[Theorem 2.6]{llr}.

\section*{Acknowledgements}  We thank Vladimir Kadets and Dirk Werner for fruitful conversations. Part of the research of this paper was developed during several visits of the first author to Instituto de Ciencias Matem\'aticas in June 2018 and to Facultad de Matem\'aticas of the Universidad Complutense de Madrid in February 2019. He is grateful to both institutions for hospitality and excellent working conditions during both visits.

\end{document}